\documentclass[twoside,11pt]{amsart}

\def\dd{\delta}

\def\GG{\Gamma}
\def\AA{\alpha}

\def\DD{\Delta}

\def\Ep{\varepsilon}
\def\0{\varnothing}

\def\RR{\mathbb R}
\def\NN{\mathbb N}

\def\ds{\displaystyle}
\def\wt{\widetilde}

\def\am{\mathrm{argmin}}

\def\Sumd{\sum_{j=1}^d}
\def\Sum{\sum_{i=1}^n}

\def\SumN{\sum_{i=1}^{N}}
\def\Su2mN{\sum_{i=2}^{N}}
\def\Su2m{\sum_{i=2}^{N}}

\def\Ex{\mathbb{E}}

\def\X{\mathcal{X}}

\def\P{\mathcal{P}}
\def\Y{\mathcal{Y}}

\def\S{\mathcal{S}}

\def\1{{\boldsymbol 1}}

\def\ol{\overline}

\theoremstyle{definition}
 
\newtheorem{lemma}{Lemma}
\newtheorem{theorem}{Theorem}
\newtheorem{definition}{Definition}
\newtheorem{example}{Example}

\usepackage{xcolor}
\usepackage{graphicx}
\usepackage{amssymb}
\usepackage{epstopdf}
\usepackage{mathtools}
\usepackage{url}
\usepackage{wrapfig}
\usepackage{caption}
\usepackage[margin =1in]{geometry}

\usepackage[linesnumbered,ruled,vlined]{algorithm2e}

\usepackage{color}

\begin{document}
\begin{center}
 \bf \Large Optimality of the Subgradient Algorithm in the Stochastic Setting
\end{center}$ $\\

\noindent {\bf Daron Anderson} \hfill {\sc andersd3@tcd.ie} \\ 
\noindent {\it Department of Computer Science and Statistics}\\
\noindent {\it Trinity College Dublin}\\
\noindent {\it Ireland}\\

\noindent {\bf Douglas Leith} \hfill {\sc doug.leith@scss.tcd.ie}\\ 
\noindent {\it Department of Computer Science and Statistics}\\
\noindent {\it Trinity College Dublin}\\
\noindent {\it Ireland}\\

      \vspace{-5mm}
      \begin{center}
      \noindent \large October 2019 
      \end{center}\vspace{1mm}

      \section*{abstract}
      	\openup .2em
      		
      		\noindent  We show that the Subgradient algorithm is universal for online learning on the simplex in the sense that it simultaneously achieves  $O(\sqrt N)$ regret for adversarial costs and $O(1)$ pseudo-regret for i.i.d costs.  To the best of our knowledge this is the first demonstration of a universal algorithm on the simplex that is not a variant of Hedge.   Since Subgradient is a popular and widely used algorithm our results have immediate broad application.   
      	%
      	%

      
      
      \section{Introduction}
       	\setlength{\parskip}{0em}
      	
      	In this paper we show that the Subgradient algorithm is universal for online learning on the simplex in the sense that it achieves $O(\sqrt{N})$ regret for adversarial sequences and $O(1)$ pseudo-regret for i.i.d sequences.  This complements a recent result by \cite{OptimalHedge} showing that the Hedge (Exponential Weights) algorithm is also universal in the same sense.  These two results are: (i) significant and interesting because the Subgradient and Hedge algorithms are popular and widely used so improved results have immediate broad application, and (ii) surprising because earlier lines of research on universal algorithms required the development of complicated algorithms purpose-built to be universal, whereas Subgradient and Hedge \cite{Kivinen} are simple and predate this line of research.  Our subgradient analysis is additionally interesting because: (i) it requires the development of a new method of proof that may be of wider application, and (ii) highlights fundamental differences between the lazy and greedy variants of Subgradient when it comes to universality, namely lazy variants are universal whereas greedy variants are not.
      	
      	The setup we consider is standard.    Let $b_1,b_2,\ldots \in \RR^d$ be a sequence of cost vectors.  On turn $n$ we know $b_1,\ldots, b_{n-1}$ (i.e. this is the full information rather than the bandit setting) and must select an action $x_n$ in the compact convex domain $\X \subset \RR^d$ with a mind to minimising the sum $\sum_{i=1}^N b_i \cdot x_i$.    The {\it regret} with respect to action $x^*\in\X$ is $\sum_{i=1}^N b_i \cdot (x_i -  x^*)$.  It is well known that when $b_1,b_2,\ldots$ are chosen by an adversary the Subgradient and Hedge algorithms (as well as others) have order $O(\sqrt N)$ regret for all $x^* \in \X$ simultaneously.   When the sequence of cost vectors is i.i.d we denote them by $a_1,a_2,\ldots$ to avoid confusion.   In the i.i.d case it is common to only consider $\X=\S$, where $\S$ is the simplex, and to bound the {\it pseudo-regret} $\Ex \left[  \sum_{i=1}^N a \cdot (x_i -   x^*) \right]$ for $a= \Ex[a_n]$ and all $x^* \in \S$.  Algorithms are known (see below for further discussion) that give $O(1)$ pseudo-regret for bounded i.i.d cost vectors.  
      	
      	In this paper we show that the lazy, anytime variant of the Subgradient algorithm has pseudo-regret at most $O(L_2^2 / \DD)$ for i.i.d cost vectors satisfying $\|a_n\|_2 \le L_2$, where  $\DD = \min \{\DD_j: \DD_j >0\}$ is  the {suboptimality gap} and $\DD_j = a \cdot (e_{j^*}- e_j)$ for $j^* \in \arg \min \{a\cdot e_j: j=1,2,\ldots, d\}$ and $e_i$ the vector with $i$'th component 1 and all others 0.  Subgradient is already known to have $O(L_2\sqrt N)$ adversarial regret.   That is, the same Subgradient algorithm simultaneously achieves good performance for adversarial loss sequences and for i.i.d sequences.
      	
      	\subsection{Related Work} 
      	
      	In recent years there has been much interest in universal algorithms, mainly in the bandit setting.  For example \cite{AnOptimal} give a randomised algorithm that simultaneously achieves $O(\sqrt{d N})$ pseudo-regret in the antagonistic case and $O(\log(N)/\DD)$ pseudo-regret in the i.i.d case. These bounds  are the same order as  the familiar Exp3.P and UCB algorithms \cite{Purple2} respectively.  See \cite{OnePractical,AnOptimal,NearlyOptimal,ImprovedAnalysis,MoreAdaptive} and references therein for more details.  All of these universal algorithms resemble Hedge in using potentials that are {\it infinitely steep} at the boundary of the simplex.
      	
      	Another line of work looks at combining algorithms  for the two  settings to obtain a universal meta-algorithm.  One strategy  is to start off with an algorithm suited to stochastic costs and then switch irreversibly to  an adversarial algorithm if evidence accumulates that the data is non-stochastic.  The other main strategy is to use reversible switches with the decision as to which algorithm (or combination of algorithms) is used being updated in an online manner.   One such strategy is ($A,B$)-Prod proposed by \cite{ExploitingEasyData}.  For  combining two algorithms $A$ and $B$ with regret $R_A$ and $R_B$ the meta-algorithm has regret at most min$\left \{R_B+ 2 \log 2, R_A + O(\sqrt{N\log N})\right\}$.  Choosing algorithm $A$ to have $O(\sqrt{N\log N})$ adversarial regret (or better) and algorithm $B$ to have $O(1)$ regret when the costs are i.i.d therefore means that the combined algorithm has $O(\sqrt{N\log N})$ regret when costs are adversarial and $O(1)$ regret when costs are i.i.d. Of course $O(\sqrt{N\log N})$ is much worse than the $O(\sqrt{N})$ adversarial regret of algorithms such as Hedge and Subgradient.  We also note that ($A,B$)-Prod uses the Prod algorithm which is equivalent to Hedge with a second-order correction.   
      	
      	A related line of work uses the fact that algorithms such as Hedge can achieve good regret if the step size is tuned to the setting of interest.  The approach taken is therefore to try to select the step size in an online fashion, see for example \cite{erven_adaptive_2011}.    With regard to the impact of step size on performance, \cite{FTLBall} consider the performance of the FTL algorithm with i.i.d costs, the FTL algorithm being equivalent to lazy Subgradient with step-size $1$.   They show that for i.i.d costs for which the mean $a$ has a unique minimiser and $\|a_n\|_\infty \le L_\infty$ the pseudo-regret of FTL on the simplex (in fact, for any polyhedron) is $O(L_\infty^3d/r^2)$, where $r$ is essentially the size of the ball around mean cost $a$ within which the minimizer is unique. This is one of the few results on Subgradient performance for i.i.d losses.   Note, however that FTL has $O(N)$ regret for adversarial costs and must be incorporated into a meta-algorithm to account for that case.
      	
      	In the foregoing work the search for universality has entailed the development of new algorithms, almost all of which are variations on Hedge.   Recently, a striking result by \cite{OptimalHedge} estabished that in the full information setting this is unnecessary.  The standard Hedge algorithm, without modification, simultaneously achieves  $O(L_\infty\sqrt N)$ regret in the adversarial case and $O(L^2_\infty\log(d)/\DD)$ pseudo-regret in the i.i.d case for bounded costs $\|a_n\|_\infty \le L_\infty$.   This is appealling both because of the simplicity and popularity of the Hedge algorithm and because of the tight nature of the bounds i.e. there is no need to pay for $O(1)$ i.i.d pseudo-regret by suffering $O(\sqrt{N \log N})$ adversarial regret.    It also raises the question as to whether the other main class of widely used algorithms, namely Subgradient, is in fact also universal.  
      	
      	\subsection{ Results and Contribution}
      	
      	Our Theorem \ref{T2} says that lazy, anytime Subgradient has pseudo-regret  $O(L_2^2 /\DD)$ in the i.i.d case, where $L_2$ bounds the $2$-norm of the cost vectors.   It follows that this variant of Subgradient simultaneously achieves  $O(L_2\sqrt N)$ regret in the adversarial case and $O(L_2^2 /\DD)$ pseudo-regret in the i.i.d case for bounded costs $\|a_n\|_2 \le L_2$.  To the best of our knowledge this is the first demonstration of a universal algorithm on the simplex that is not a variant of Hedge.   Since Subgradient is a popular and widely used algorithm our results have immediate broad application.   
      	
      	The method of proof of Theorem  \ref{T2} appears to be new.  Rather than follow a sequence of actions inside the simplex, we follow the sequence of unprojected actions, and show the sequence eventually passes with high probability into the normal cone of the optimal vertex. Hence the projected action eventually snaps to the correct vertex.  This behaviour, whereby Subgradient converges to the optimal action in finite-time, is qualitatively different from Hedge-type algorithms where the actions only approach the optimal vertex asymptotically.   This new method of proof is likely to be of wider application.
      	
      	A technical tool used that seems new in the context of Online Optimisation is the vector concentration inequality Theorem 3.5 of \cite{GoodAH}.   For comparison it is possible to get $O\left (\Sumd{L_2^2/\DD_j} \right )$ pseudo-regret bounds for Subgradient  using only scalar concentration inequalities for each component, and to obtain a $O\left ( \log(d)L_2^2/\DD \right )$ bound by using the adversarial bound over an initial segment of turns and then a probabilistic bound over the remainder.    However the  Pinelis vector inequality allows us to tighten these bounds to to the dimension-free $O(L_2^2 /\DD)$.  Removing the $\log(d)$ factor is a significant improvement when $d$ is large.   
      	
      	Theorem \ref{tail} extends our analysis to include tail bounds on the pseudo-regret.  Namely, for Subgradient there is $c > 0$ and $C>0$ independent of $\eta,\DD$ with   $$P\left(\SumN a \cdot (x_i-x^*)    > c + \frac{L_2^2}{\DD} \dd \right) \le O\big(   e^{-C\dd}\big )$$ for all $\dd$ sufficiently large. 
      	
      	
      	One advantage of Subgradient is it can be applied with actions on arbitrary domains $\X$, not just the simplex $\S$. In Section \ref{notsimplex}, however, we show this can break the results of Theorem \ref{T2}.  Namely, for each $\Ep > 0$ there is a domain and i.i.d cost vectors that give pseudo-regret $\Omega(N^{1/2- \Ep})$.  Thus the i.i.d pseudo-regret can be almost as bad as the $O(\sqrt N )$ worst-case regret. These domains have the form $\{(x,y) \in \RR^2: y \ge x^\AA\}$ for $\AA > 2$ and are not strictly convex at the origin.  
      	In Section \ref{greedynotuniversal} we show the use of lazy rather than greedy Subgradient  is important in achieving universal performance. We give an example that shows greedy Subgradient is {\it too sensitive} to adapt to the i.i.d setting.

\section{Terminology and Notation}

\noindent Throughout $d$ is the dimension of the online optimisation problem.
We write $x(j)$ for the components of $x \in \RR^d$ and $e_1,e_2,\ldots, e_d \in \RR^d$ for the coordinate vectors and $\1$ for the vector $(1 , \ldots, 1 ) \in \RR^d$. Define the $d$-simplex $\S = \{x \in \RR^d: \mbox{all } x(j) \ge 0  \mbox{ and } \1 \cdot x   =1\}$. 

For any function $f: X \to \RR$ we write $\am\{f(x):x \in X\}$ for the set of minimisers. Each linear function on the simplex is minimised on some vertex. Hence $\min\{a \cdot x: x \in \S\} = \min\{ a \cdot e_j: j \le d\}$. We write $\|\cdot\|$ for the Euclidean norm and for any convex $\X \subset \RR^d$ we write $P_\X(x) = \am\{\|y-x\|^2: y \in \X\}$ for the Euclidean projection of $x$ onto $\X$. 

Thoughout the {\it cost vectors} $a_1,a_2,\ldots \in \RR^d$ are realisations of a sequence of i.i.d random variables with each $\Ex[a_i]=a$. When we write $b_1,b_2,\ldots $  we make no assumptions on whether the cost vectors are i.i.d or otherwise. We assume bounds of the form $\|a_i - a\| \le R$ and $\|a_i\| \le L$.  

For cost vectors $b_1,b_2,\ldots$  the {\it regret} of an action sequence $x_1,\ldots, x_N$ is defined as \mbox{$\SumN b_i \cdot (x_i  -   x^*)$} for $x^* \in \am  \SumN b_i \cdot x$.
For stochastic cost vectors $a_1,a_2,\ldots$  the {\it pseudo-regret} of the action sequence is \mbox{$ \Ex \left [\SumN a \cdot (x_i - x^*) \right ] $} for $x^* \in \am  \, a \cdot x$. Here the expectation is taken over the domain of $a_1,\ldots, a_N$. 



By permuting the coordinates if neccesary we assume $e_1$ is a minimiser of $a$ and that the differences $\Delta_j = a \cdot(e_j-e_1)$ satisfy $0=\DD_1 \le \DD_2 \le \ldots \le \DD_d$. 
The permutation is part of the analysis only, and our algorithm does not require access to it. We write $\DD=\DD_2 = \min\{\DD_j: \DD_j >0\}$.



\section{Pseudo-Regret}

\noindent The subgradient algorithm is one of the simplest and most familiar algorithms for online convex optimisation. The anytime version Algorithm 1 does not need the time horizon in advance. In this algorithm the step size on turn $n$ is $\eta/\sqrt{n-1}$ where $\eta >0$ is a design parameter.
 

\begin{algorithm}[!h]\caption{Anytime Subgradient Algorithm}
	\DontPrintSemicolon 
	\KwData{Compact convex subset $\X \subset \RR^d$. Parameter $\eta > 0$.\;}
	$\text{select action } x_1=P_\X(0)$\; 
	$\text{pay cost } \ds a_1 \cdot x_1$\;
	\SetKwBlock{Loop}{Loop}{EndLoop}
	\For{$n=2,3,\ldots$}{
	
	\text{recieve} $a_{n-1}$\;
	$\ds y_n = -\eta \left( \frac{a_1 + \ldots + a_{n-1}}{\sqrt {n-1}} \right )$\;
	$\text{select action } \ds x_n = P_\X(y_n)$\;
	$\text{pay cost } \ds a_n \cdot x_n$
}
	
\end{algorithm} 

The subgradient algorithm is known to have $O(L\sqrt N)$ regret. See \cite{Purple1} and \cite{Zinkevich}.

\begin{theorem} \normalfont\label{worstcase}
	For cost vectors $b_1.b_2,\ldots, b_N$ with all $\|b_i\| \le L$ Algorithm 1 with parameter $\eta>0$ has regret satisfying
	\begin{align*}
	\sum_{i=1}^N b_i \cdot (x_{i} - x^*) \le   LD + \left (\frac{1}{2 \eta}\|\X\|^2 +  2\eta L^2_2 \right)  \sqrt {N} \end{align*}
	for $\|\X\| = \max\{\|x\|: x \in \X\}$ and $D$ the diameter of $\X$. In particular for $\X=\S$ and $\eta = 1/2L$ we have
	\begin{align*}
	\sum_{i=1}^N b_i \cdot (x_{i} - x^*) \le \sqrt 2 L + 2 L  \sqrt N   .\end{align*}
\end{theorem}

\begin{proof} See Appendix A. \end{proof}
Our main result is that, in addition to the above bound, the algorithm adapts to the stochastic case to have $O( L^2_2/\DD)$ pseudo-regret. In particular the bound is independent of the dimension of the problem.

\begin{theorem} \normalfont\label{T2}Suppose the cost vectors $a_1,a_2,\ldots$ are independent  with all $\|a_i\| \le L_2$ and $\|a_i- a\| \le R_2$. Then Algorithm 1 run on the simplex has pseudo-regret at most \begin{align}\label{mainT}   \Ex\left [	\sum_{i=1}^\infty a \cdot(x_i -x^*)\le  \right]  \sqrt 2 L + \frac{(1   +  2\eta^2 L^2_2 ) L}{6} + 
\frac{3/\eta^2 + 6L^2_2 + 72R^2_2 e^{-1/2\eta^2 R^2}}{\DD}	.
	\end{align}
	
	for $\DD = \min\{\DD_j: \DD_j >0\}$.  In particular for $\eta = 1/2L$ the pseudo-regret is at most\begin{align*}   \Ex\left [	\sum_{i=1}^\infty a \cdot(x_i -x^*)\right] \le     2 L +    \frac{18L_2^2 + 72R_2^2}{\DD} 
	\end{align*}
\end{theorem}
  The strategy is to use Theorem 1 over an initial segment of the turns and a probabilistic bound over the final segment. Over that segment we are interested in conditions that make $-\frac{\eta}{\sqrt n} \Sum a_i$ project onto the convex hull of $\{e_1,\ldots, e_k\}$ as this ensures the regret is at most $\DD_k$. To that end we use the following lemma that is proved in the Appendix.


\begin{lemma} \normalfont\label{goal} Suppose $w \in \RR^d$ has two coordinates $k,\ell$ with $w_k - w_\ell \ge 1$. Then $P_\S(w)$ has $\ell$-coordinate zero. 
\end{lemma}

Now we show how smaller errors make us select better vertices.

\begin{lemma} \label{hull}\normalfont Suppose $n \ge 9/\DD_j^2 \eta^2$. Then for $\big \| \frac{1}{n}\Sum (a-a_i) \big \|_\infty \le    \DD_j/3$ the action $x_{n+1}$ is in the convex hull of $e_1,\ldots, e_{j-1}$ and the pseudo-regret for that round is at most $\DD_{j-1}$.
\end{lemma}

\begin{proof}Since  $x_{n+1} = P_\S \left(-\frac{\eta}{\sqrt n} \Sum a_i \right )$ the previous lemma says it is enough to show  for $\ell \ge j $ that $\frac{\eta}{\sqrt n} \Sum a_i(\ell) - \frac{\eta}{\sqrt n} \Sum a_i(1) \ge 1$. To that end write 
	\begin{align*}\frac{\eta}{\sqrt n} \Sum \big(a_i(\ell)- a_i(1) \big) = \frac{\eta}{\sqrt n} \Sum \DD_\ell +\frac{\eta}{\sqrt n} \Sum \big(a_i(\ell)- a(1) \big)+\frac{\eta}{\sqrt n} \Sum \big(a(1)- a_i(1) \big)\\
	\ge  \eta \DD_\ell \sqrt n  -\frac{2\eta}{\sqrt n} \Big \|  \Sum (a-a_i) \Big \|_\infty\ge  \eta  \DD_j \sqrt n -\frac{2\eta \DD_j}{3} \sqrt n   =  \frac{ \eta \DD_j}{3} \sqrt n    
	\end{align*}  
	The assumption on $n$ makes the right-hand-side at least $1$. 
\end{proof}


 
 Now we prove our bound over the final segment.
 
\begin{lemma} \normalfont Suppose $a_1,a_2,\ldots$ have $\|a_i-a\|_\infty \le R_2$. Then for $n_0 > \lceil 9/\DD^2 \eta^2\rceil $ Algorithm 1 gives 
	\begin{align*}\Ex\left [	\sum_{i=n_0}^\infty a \cdot(x_i -e_1) \right]\le \frac{72 R_2^2}{\DD}\exp\left(-\frac{1 }{2\eta^2 R_2^2 }\right).
	\end{align*}
	\end{lemma}

\begin{proof} Write the distinct elements of $\{\DD_2,\DD_3,\ldots, \DD_d\}$ in increasing order as $\DD(2)<  \ldots< \DD(K)$ for some $K \le d$.	Define each $\GG(j) = \DD(j)^2 /18 R^2_2$. Theorem says each $P \big( \frac{1}{n}\big \|\sum_{i=1}^n(a_i - a) \big\| \ge \DD(j)/3 \big) \le 2\exp ^{-\GG(j) n}.$ Since $\|x\|_\infty \le \|x\|$ we can combine Lemmas \ref{goal} and \ref{hull} for $n \ge n_0$ to bound the complementary CDF:  
	
	$$P\big(a \cdot (x_{n+1}-e_1) >x\big ) \le  \begin{cases} 
	\ds 2e^{-\GG(2) n}  & 0 <  x \le \DD(2)\\
	\ds 2e^{-\GG(k) n}  & \DD(k-1) < x \le \DD(k)  \text{ with } k\ge  3\\[5pt]
	\ds 0 &   \DD(K) < x
	\end{cases}
	$$
	Lemma \ref{CDF} lets us integrate the piecewise function to get  $\Ex[a \cdot (x_{n+1}-e_1)] \le 2 \DD(2)   e^{-\GG(2) n} + 2\sum_{k=3}^d (\DD(k) -\DD(k-1))e^{-\GG(k) n}$. Now sum over  $n$ and observe, since  the summands are decreasing, the sums are bouded by the integrals:
	
	\begin{align} \Ex\left [	\sum_{i=n_0}^\infty a \cdot(x_i -e_1) \right] \le 2 \DD(2)   \sum_{i=n_0}^\infty e^{-\GG(2) n} + 2\sum_{k=3}^d \sum_{i=n_0}^\infty(\DD(k) -\DD(k-1))e^{-\GG(k) n} \label{stabilises}\\
	\le 2 \DD(2)   \int_{i=n_0}^\infty e^{-\GG(2) n} + 2\sum_{k=3}^d \int_{i=n_0}^\infty(\DD(k) -\DD(k-1))e^{-\GG(k) n}\notag \\
	\le \left(\frac{2 \DD_2}{\GG(2)} + 2 \sum_{k=3}^d \frac{\DD(k) -\DD(k-1)}{\GG(k)}\right) e^{-\GG(2)n_0}\notag\\=  36R_2^2\left(\frac{1}{\DD_2} +   \sum_{k=3}^d \frac{\DD_{k} -\DD(k-1)}{\DD(k)^2}\right) e^{-\GG(2)n_0}\notag   
	\end{align}
	
	To bound the above use the integral inequality $\int_a^b f(x) dx \ge (a-b) \min\{f(x): a \le x \le b\}$. For $f = 1/x^2$ we get $\frac{\DD_{k} -\DD_{k-1}}{\DD_k^2} \le \int_{\DD_{k-1}}^{\DD_k} \frac{dx}{x^2}$. Hence the above sum is at most
	
	\begin{align*} \sum_{k=3}^d \frac{\DD_{k} -\DD_{k-1}}{\DD_k^2}\le \sum_{k=3}^d  \int_{\DD_{k-1}}^{\DD_k} \frac{dx}{x^2} =\int_{\DD_{2}}^{\DD_d} \frac{dx}{x^2}\le  \int_{\DD_{2}}^{\infty} \frac{dx}{x^2} = \frac{1}{\DD_2}.
	\end{align*}and we get $  \Ex\left [	\sum_{i=n_0}^\infty a \cdot(x_i -e_1) \right]   \le  \frac{72  R^2_2}{\DD_2}  e^{-\GG(2)n_0}$. Recall the definitions of $n_0$ and $\GG(2)$ to see the exponent is at least $\frac{\DD^2}{18 R^2_2}\frac{9}{\DD^2 \eta^2} = \frac{1}{2\eta^2 R_2^2}$.
	\end{proof}

	\noindent{\it Proof of Theorem \ref{T2}.} 
	For $\X$ the simplex $\|X\|^2=2$ and $D=\sqrt 2$. Hence for $n_0 = \lceil 9/\DD^2 \eta^2\rceil $ Theorem \ref{T2} gives the regret bound 
		\begin{align*}
	\sum_{i=1}^{n_0} a_i \cdot (x_{i} - x^*) \le   \sqrt 2 L + \left (\frac{1}{  \eta}  +  2\eta L^2_2 \right)  \sqrt {n_0}\le   \sqrt 2 L +\left (\frac{1}{  \eta}  +  2\eta L^2_2 \right)  \sqrt {1+9/\DD^2 \eta^2 }  \end{align*} 
	By concavity the square root is at most
	\begin{align*}  \sqrt {1+9/\DD^2 \eta^2 } \le  \sqrt { 9/\DD^2 \eta^2 } + \frac{1}{2} \sqrt{\frac{\eta^2 \DD^2}{9}} = \frac{3}{\eta \DD} + \frac{\eta \DD}{6} \le  \frac{3}{\eta \DD} + \frac{\eta L}{6}  \end{align*} 
	and we get $\ds \sum_{i=1}^{n_0} a_i \cdot (x_{i} - x^*) \le    \sqrt 2 L + \frac{(1   +  2\eta^2 L^2_2 ) L}{6} + \left(\frac{1}{\eta^2} + 2L^2\right) \frac{3}{\DD}$.
	Hence the same bound holds for the expected regret. Since the pseudo-regret is always less than expected regret we can combine the above with the previous lemma to complete the proof.    \begin{flushright} $\qed$ \end{flushright}

As mentioned in Section 2 our bound has different behaviour to that of  \cite{OptimalHedge} for  Hedge, and is more appropriate if the cost vectors come from a sphere rather than a cube.  On the other hand our bound is dimension-independent.

\subsection{Better Constants}

\noindent Theorem \ref{T2} can be improved by replacing the constants $R_2,L_2$ with the smaller constants that arise when we ignore the components of the cost vectors that are perpendicular to the simplex.

\begin{definition}\normalfont \label{tilde}
	Let $P: \RR^d \to V$ be the projection onto the convex hull $V = \{x \in \RR^d: \sum_{j=1}^d x(j)=0\}$ of the simplex. Define $\wt L_2 = \sup \{\|Pa_n\|: n \in \NN\}$ and $\wt R_2 = \sup \{\|P(a_n-a)\|: n \in \NN\}$.
\end{definition}

To write down $\wt L_2$ and $\wt R_2$ explicitly recall the Euclidean norm can be computed with respect to any orthonormal basis. Hence we can choose an orthonormal basis $u_1,\ldots, u_{d-1}$ for $V$ and then add  $u_d = \frac{1}{\sqrt d} \1$ to get an orthonormal basis for the whole space. The projection of each $x= \Sumd c_j u_j$ onto $V$ is just $\sum_{j=1}^{d-1}c_j u_j$ and the norm of the projection is $$\textstyle\|Px\|= \sqrt{\sum_{j=1}^{d-1}c_j^2}=  \sqrt{\sum_{j=1}^{d}c_j^2 - c_d^2 }  =  \sqrt{\sum_{j=1}^{d}c_j^2 - (u_d \cdot x)^2}= \sqrt{\|x\|^2 - \frac{1}{ d} (\1 \cdot x)^2} .$$
Hence we can write
\begin{align}\label{tildexplicit} \textstyle \wt L^2_2 = \sup \left\{ \|a_n\|^2 - \frac{1}{d} \Big(\Sumd a_n(j) \Big)^2: n \in \NN \right\}\\\textstyle \wt R^2_2 = \sup \left\{ \|a_n-a\|^2 - \frac{1}{d} \Big(\Sumd a_n(j) -a(j)\Big)^2: n \in \NN \right\}\notag
\end{align}

\begin{theorem} \normalfont \label{T'}Theorems 1 and 2 hold with the constants $L$ and $R$ replaced with $\wt L_2$ and $\wt R_2$. 
\end{theorem}

\begin{proof}We claim the actions given any cost vectors $b_1,b_2,\ldots$ are the same as those given the projections $P b_1, P b_2,\ldots$. From line 5 of Algorithm 1 we see it is enough to show for each $x \in \RR^d $ that $P_\S(x) = P_\S(P_V(x))$. To that end   consider the sphere $S$ with centre $x$ and radius $\|x-P_\S(x)\|$. This sphere meets the simplex at the single point $P_\S(x)$. The intersection $S \cap V$ is a circle centred at $P_V(x)$ that meets the simplex at the point $P_\S(x)$. It follows $P_\S(x)$ is the projection of $P_V(x)$ onto the simplex as required.

It follows $a_1,a_2,\ldots$ give the same actions as $Pa_1,Pa_2,\ldots$. Hence the  bounds in Theorems 1 and 2 hold with $\wt L_2,\wt R_2$ in place of $L,R$ and $\Ex \left[\sum_{i=1}^\infty Pa \cdot(x_i-x^*)\right]$ on the left for each $x^* \in \S$. To complete the proof we claim $Pa \cdot(x_i-x^*) = a \cdot(x_i-x^*)$. This is equivalent to $(Pa-a) \cdot(x_i-x^*)=0$ which holds because $Pa-a $ is perpendicular to $V$ and $x_i-x^*$ is contained in $V$.\end{proof} 

\section{Counterexamples}\label{counter}

\noindent One shortcoming of Hedge-type algorithms is they only make sense when the action set is the simplex.  This is because they use potentials that are {\it infinitely steep} on the boundary. On the other hand the quadratic potential from Subgradient is defined everywhere and the algorithm can be applied to arbitrary action sets. 
This raises the question of what kinds of domains we can use to replace the simplex while keeping the order bounds  from Theorems  \ref{T2} and \ref{T'}. 

\subsection{Beyond the Simplex}\label{notsimplex}
\noindent Here we give an example of a curved domain where the order bounds in  Theorems  \ref{T2} and \ref{T'} fail.
\begin{example}\normalfont 
	Suppose we run Algorithm $1$ with parameter $\eta=1$ and  domain $$\Y = \{(x,y) \in \RR^2: y \ge |x|^3 \text{ and } x \le 1\}.$$ There is  a sequence $a_1,a_2,\ldots$ of i.i.d cost vectors such that  
	\begin{align*}
	\Ex \left[\sum_{i=1}^n a \cdot (x_{i} - x^*) \right]\ge \Omega(\sqrt[4] n) .\end{align*}
\end{example}		
		
\begin{proof}Let the cost vectors be $a_n=(B_n,1)$ for $B_1,B_2,\ldots$ independent with each $P(B_i=1) = {P(B_i=-1)}=1/2$. The central limit theorem says $\frac{\eta}{\sqrt n} \Sum B_i$ tends to a normal distribution. Hence there are $m \in \NN$ and $c >0$ such that for all $n \ge m$ we have $P\big(\frac{\eta}{\sqrt n} \Sum B_i > 1\big ) \ge c$.
We claim that if $\frac{\eta}{\sqrt n} \Sum B_i > 1$ occurs then $a \cdot(x_{n+1}-x^*) \ge  6^{-3/2} n^{-3/4}$. Hence we have 

\begin{align*}
\Ex \left[\sum_{i>M}^n a \cdot (x_{i} - x^*) \right]\ge   c\sum_{i>m}^n 6^{-3/2} i^{-3/4} \ge  c6^{-3/2} \int_{m+1}^{n}  x^{-3/4} dx  \ge  \frac{c6^{-3/2}}{4} \big(\sqrt[4] {n}- \sqrt[4] {m+1} \big). 
\end{align*}

It follows the pseudo-regret is at least $-2\|a\| m + \frac{c6^{-3/2}}{4} \big(\sqrt[4] {n}- \sqrt[4] {m+1} \big) \ge \Omega\big (\sqrt[4]n \big )$.

To prove the claim suppose $\frac{\eta}{\sqrt n} \Sum B_i > 1$  and write $x_{n+1}=(x,|x|^3)$. Since  $a=(1,0)$ has minimiser $x^* = (0,0)$ we have $a\cdot(x_{n+1}-x^*) = x^3$.
 Since $x_{n+1}$ is the projection of $y_{n+1} = -\frac{1}{\sqrt n} \Sum a_i $ onto $\X$ we have either  (a) $x_{n+1} = (\pm 1,1)$  or (b) $x_{n+1}$ is the projection of $y_{n+1}$ onto the graph $y=|x|^3$. In the first case $a\cdot(x_{n+1}-x^*) = 1$. 
 
 In the second case  $y_{n+1}$ is in the left quadrant and so $x <1$. Since $x_{n+1}$ is the projection of $y_{n+1}$ we know $ y_{n+1}-x_{n+1}$ is outward normal to the graph. Expand the definition to see 
\begin{align*}
 y_{n+1}-x_{n+1}=-\frac{1}{\sqrt n} \Sum a_i - x_{n+1} = \left( -\frac{1}{\sqrt n} \Sum B_i-x, -\sqrt n-x^3\right)  
\end{align*}
Since $x<1$ the slope at $x_{n+1}$ is $-3x^2$ . Hence the outwards normal points along  $(-3x^2,-1)$. Rescale to see 
\begin{align*}
 \frac{1}{\sqrt n} \Sum B_i+x   = 3x^2(\sqrt n + x^3) \implies 1< 3x^2(\sqrt n + x^3) <  6x^2 \sqrt n  
\end{align*}where we have used $|x| \le 1$. The above implies $x \ge 6^{-1/2} n^{-1/4}$ and so $a \cdot(x_{n+1}-x^*) = |x|^3 \ge  6^{-3/2} n^{-3/4}$.  This completes the proof.
\end{proof}
	
More generally we can take the domain $\Y_\AA=\{(x,y) \in \RR^2: y \ge  |x| ^\AA \text{ and } x \le 1\}$ for any $\AA > 2$. Then an analogous proof  to the above shows the pseudo-regret has order $\Omega(N^{1- \frac{\AA}{2(\AA-1)}})$. Hence we get the following

\begin{lemma} \normalfont\label{almost}
	Let $\Ep >0$ be arbitrary. There exists a compact convex domain $\Y \subset \RR^2$ and i.i.d sequence $a_1,a_2,\ldots$ of  cost vectors  with $\Ex[a_i]=a$ such that running  Algorithm $1$ with any parameter $\eta=1$ gives  
	\begin{align*}
	\sum_{i=1}^n a \cdot (x_{i} - x^*) \ge \Omega(n^{1/2-\Ep}) .\end{align*}
\end{lemma}

On the other hand  \cite{FTLBall} show we can get $O(\log N)$ regret against i.i.d cost vectors on each $\Y_\AA$  provided the minimiser is not the origin. Their Theorem 3.3 says that since $f(x) =|x|^\AA$ has nonzero second derivative away from the origin we can get $O(\log n)$ regret by running Follow-the-Leader. Likewise since $f(x)=x^2$ has nonzero derivative everywhere the same theorem gives a $O(\log n)$ bound on $\Y_2$ for any minimiser.

Lemma \ref{almost} says that as $\AA \to \infty$ the worst-case behaviour of Subgradient approaches $\Omega(\sqrt n)$. It is interesting that for $\AA=\infty$ the domain is the box $[-1,1] \times [0,1]$, and a similar argument to Theorem \ref{mainT} says Subgradient gives $O(1)$ regret over the box.

\noindent  
\subsection{Greedy Subgradient is not Universal}\label{greedynotuniversal} \noindent The fact that Theorems \ref{T2} and \ref{T'} are proved for the Lazy Subgradient algorithm rather than  Greedy Subgradient is important. Indeed the theorems fail if we instead use the greedy version. The reason is that Greedy is too sensitive to next cost vector to remain on the optimal vertex.

Recall the greedy Subgradient on domain $\X$ chooses actions $x_2,x_3\ldots$ recursively by $y_{n+1} = x_n - \frac{\eta}{\sqrt n} a_n$ and $x_{n+1} = P_\X(y_{n+1})$. It is straightforward to come up with i.i.d examples where the pseudo-regret is $\Omega(\sqrt N)$.  This matches the worst-case bound for regret \cite{Zinkevich}.

\begin{example}\normalfont 
	Suppose we run the greedy Subgradient algorithm on the $2$-simplex with parameter $\eta=1$. There is a sequence $a_1,a_2,\ldots$ of i.i.d cost vectors with  $\Ex[a_i]=a$ and 
	$$\Ex \left[\Sum  a \cdot(x_i - x^*)\right] \ge \Omega(\sqrt n).$$
\end{example}

\begin{proof}Let  $V = \{x \in \RR^d: \sum_{j=1}^d x(j)=0\}$ be the affine span of the simplex.  In Theorem \ref{T'} we show that for any $x \in \RR^d$ that $P_\S(x) = P_\S(P_V(x))$. Hence  running Greedy subgradient on the $2$-simplex with cost vectors $a_n=(a_n(1),a_n(2))$ is equivalent to running it on domain $[-1,1]$ with cost vectors (scalars) $\frac{a_n(1) - a_n(2)}{2}$. For ease of notation we will work in the second setting.
	
	Let the cost vectors (scalars) be $a_n=1$ with probability $3/4$ and $a_n = -1$ with probability $1/4$. Then $a = 1/2$ and the minimiser is $x^*=-1$.
	We claim each $x_{n+1} \ge -1 + 1\sqrt n $ with probability at least  $1/4$. Hence $\Ex[a\cdot (x_{n+1} - x^*)] \ge \frac{1}{4\sqrt n}$ and so
	
	\begin{align*}\Ex \left[\sum_{i=2}^{n}  a \cdot(x_i - x^*)\right] \ge \sum_{i=1}^{n-1}  \frac{1}{4\sqrt n} \ge \int_1^{n-1}\frac{dx}{4\sqrt x} = \frac{\sqrt {n-1} - 1}{2} \ge \Omega(\sqrt n).
	\end{align*}
	
	By definition $x_{n+1} = P_\X(x_n - \frac{a_n}{\sqrt n})$. Since $x_n \in [-1,1]$ we have $x_n  \ge -1$. With probability $1/4$ we have $a_n = -1$ and so  $x_n - \frac{a_n}{\sqrt n} = x_n + 1/\sqrt n \ge -1 + 1/\sqrt n$ as required. 
\end{proof}

\section{Tail Bounds}\label{tail}

In this section we show the value $\sum_{i=1}^\infty a \cdot (x_i -   e_1)$ is unlikely to stray too far from the expectation. Recall Theorem \ref{mainT} says  $\Ex \big[\sum_{i=1}^\infty a \cdot (x_i -   e_1)\big] \le O(L^2_2/\DD)$. Next we show the the probability of the coefficient being large shrinks exponentially.

\begin{theorem}\label{tail}\normalfont Suppose the cost vectors $a_1,a_2,\ldots$ are independent  with all $\|a_i\| \le L_2$ and $\|a_i- a\| \le R_2 $. Then Algorithm 1 run on the simplex has tail bound \begin{align*}
		P \left(\sum_{i=1}^\infty a \cdot (x_i -   e_1) > 2 L_2 + \frac{  L^2_2}{\DD} t  \right) \le (1+36  R^2)\exp\left(-\frac{ t}{24  R^2} \right)
		\end{align*} for all   $t \ge \frac{3}{L^2_2} \left(2L_2+ \frac{\sqrt 2}{ \eta} + \frac{\sqrt 2  }{3 }\DD  \right)^2$.
\end{theorem}

 Like before we derive separate bounds over initial and final segments. For the final segment we have the lemma.

\begin{lemma}\label{final}\normalfont For each $n >  9/\DD^2 \eta^2$ we have 
	\begin{align*}
	P \left( \sum_{i>n}^\infty a \cdot (x_i -   e_1) =0 \right) \ge 1-36R^2_2 \exp\left(-\frac{\DD^2}{18 R^2} n \right)
	\end{align*}
	\end{lemma}

\begin{proof} Combine Lemma \ref{hull} and Theorem \ref{Pinelis} to see $P( x_{i+1} \ne e_1) \le 2 \exp\left(-\frac{\DD^2}{18 R^2} i \right)$. Take a union bound to see $P(  x_{i+1} \ne e_1$ for some $i \ge n)$ is at most
		\begin{align*}
	 \sum_{i=n}^\infty 2 \exp\left(-\frac{\DD^2}{18 R^2} i \right) \le \int_{n}^\infty 2 \exp\left(-\frac{\DD^2}{18 R^2} x \right)dx = \frac{36R^2}{\DD^2 }\exp\left(-\frac{\DD^2}{18 R^2} n \right).
	\end{align*}To finish the proof observe if $x_{n+1}=x_{n+2} = \ldots = e_1$ the pseudo-regret after turn $n$ is zero.
	\end{proof}

Now we bound the initial segment.

\begin{lemma}\label{initial}\normalfont For each $n \in \NN$ and $t >0$ we have  
	\begin{align*}
	P \left(\sum_{i=1}^n a \cdot (x_i -   e_1) > 2 L_2 + (2 L_2+t)  \sqrt n  \right) \le \exp\left(-\frac{t^2}{4 R^2} \right)
	\end{align*}
	\end{lemma}

\begin{proof}By Theorem \ref{worstcase} we have
	\begin{align*}\sum_{i=1}^n a \cdot (x_i -   e_1)   = \sum_{i=1}^n a_i \cdot (x_i -   e_1) +   \sum_{i=1}^n (a-a_i) \cdot (x_i -   e_1)\\ \le \sqrt 2 L_2 + 2 L_2  \sqrt n  +   \sum_{i=1}^n (a-a_i) \cdot (x_i -   e_1)
	\end{align*}
	For the final sum Lemma \ref{martingale} says $X_i = (a-a_i) \cdot (x_i -   e_1)$ is a martingale difference sequence with respect to $a_1,a_2,\ldots$. Since each $|X_i| \le \|a-a_i\| \|x_i -   e_1\| \le \sqrt2  R$  Lemma \ref{Hoeffding} says the sum exceeds $t\sqrt n$ with probability at most $\exp\left(-\frac{t^2}{4 R^2} \right)$.
	   \end{proof}

	Now we combine the  previous two lemmas.\\
	

{\noindent \it Proof of Theorem \ref{tail}.} For each $t  \ge \max \left\{ \frac{\sqrt 2}{ \eta} , \frac{\sqrt 2  }{3 }\DD \right\}$ we have $\frac{9  t^2}{2   \DD^2} \ge  \frac{9}{\DD^2 \eta^2}$  and we can combine the previous two lemmas with $n = \left \lceil \frac{9 t^2}{2 \DD^2}   \right \rceil$. For the left-hand-side of Lemma \ref{initial} we have
	
	\begin{align*}(2L_2+t) \sqrt n \le (2L_2+t) \sqrt {  \frac{9  t^2}{2 \DD^2}   +1}  \le   (2L_2+t) \sqrt {  \frac{9  t^2}{ \DD^2} } = \frac{3 }{\DD} (2L_2+t)t
	\end{align*}where the second inequality uses $t  \ge   \frac{\sqrt 2  }{3 } \DD$ to see $1 \le  \frac{9  t^2}{2 \DD^2}$. Hence we have 
	
	$$(2L_2+t) \sqrt n \le   \frac{3 }{\DD} (2L_2+t)t = \frac{3 }{\DD} \big((t+L_2)^2-L^2_2\big ) \le  \frac{3 }{\DD} (t+L_2)^2$$ 
For the right-hand-side of Lemma \ref{final} we have 
	\begin{align*}  \exp\left(-\frac{\DD^2}{18 R^2} n \right) =  \exp\left(-\frac{\DD^2}{18 R^2} \left \lceil \frac{9  t^2}{2\DD^2}  \right  \rceil \right)\le  \exp\left(-\frac{t^2}{4 R^2}  \right) 
	\end{align*}
	Hence the two lemmas combine to give 
		\begin{align*}
	P \left(\sum_{i=1}^\infty a \cdot (x_i -   e_1) > 2 L_2 + \frac{3 }{\DD} (t+L_2)^2  \right) \le (1+36R^2)\exp\left(-\frac{t^2}{4 R^2} \right).
	\end{align*}
	Define $\dd = t+L_2$ to see  for all $\dd   \ge L_2+ \max \left\{ \frac{\sqrt 2}{ \eta} , \frac{\sqrt 2  }{3 }\DD \right\}$ that \begin{align*}
	P \left(\sum_{i=1}^\infty a \cdot (x_i -   e_1) > 2 L_2 + \frac{3 }{\DD} \dd^2  \right) \le (1+36R^2)\exp\left(-\frac{(\dd-L_2)^2}{4 R^2} \right)
	\end{align*} If in addition $\dd \ge 2L_2$ we have $(\dd-L_2)^2 \ge (\dd/2)^2 $. Hence  for all $\dd   \ge 2L_2+ \max \left\{ \frac{\sqrt 2}{ \eta} , \frac{\sqrt 2  }{3 }\DD \right\}$ we have 
	\begin{align*}
	P \left(\sum_{i=1}^\infty a \cdot (x_i -   e_1) > 2 L_2 + \frac{3 }{\DD} \dd^2  \right) \le (1+36R^2)\exp\left(-\frac{ \dd ^2}{8 R^2} \right)
	\end{align*}
	Finally define $t = 3 \dd^2/L^2_2$ and the above becomes 
	
		\begin{align*}
	P \left(\sum_{i=1}^\infty a \cdot (x_i -   e_1) > 2 L_2 + \frac{L^2_2}{\DD} t  \right) \le (1+36R^2)\exp\left(-\frac{ t}{24 R^2} \right)
	\end{align*} for all $t \ge \frac{3}{L^2_2} \left(2L_2+ \frac{\sqrt 2}{ \eta} + \frac{\sqrt 2  }{3 }\DD  \right)^2$
	 \begin{flushright} $\qed$ \end{flushright}

By the same argument as Theorem \ref{T'} we can replace the constants in Theorem \ref{tail} with the smaller constants (\ref{tildexplicit}). 
 
\section{Simulations}

\noindent Here we plot the results of some simulations. We compare the coefficients in Theorem 2 to those observed empirically. Our simulations suggest the true constants are two orders of magnitude smaller than our theoretical bounds.

For each simulation  we fix $\DD=\eta=1$. The i.i.d sequence $a_1,a_2,\ldots, \in \RR^d$ was generated as $a_n = a + R N_n$ for $N_1,N_2,\ldots$ drawn uniformly from the $(d-1)$-dimensional unit sphere. Sampling on the unit sphere was done by drawing inpendent standard normals $U_1,\ldots, U_d$ and normalising the vector $(U_1,\ldots, U_d)$. 
See \cite{UniformSphere} Section 4 for a proof of this method.

 \begin{figure}[!h]
  \begin{center} 
    \includegraphics[width=1 \textwidth]{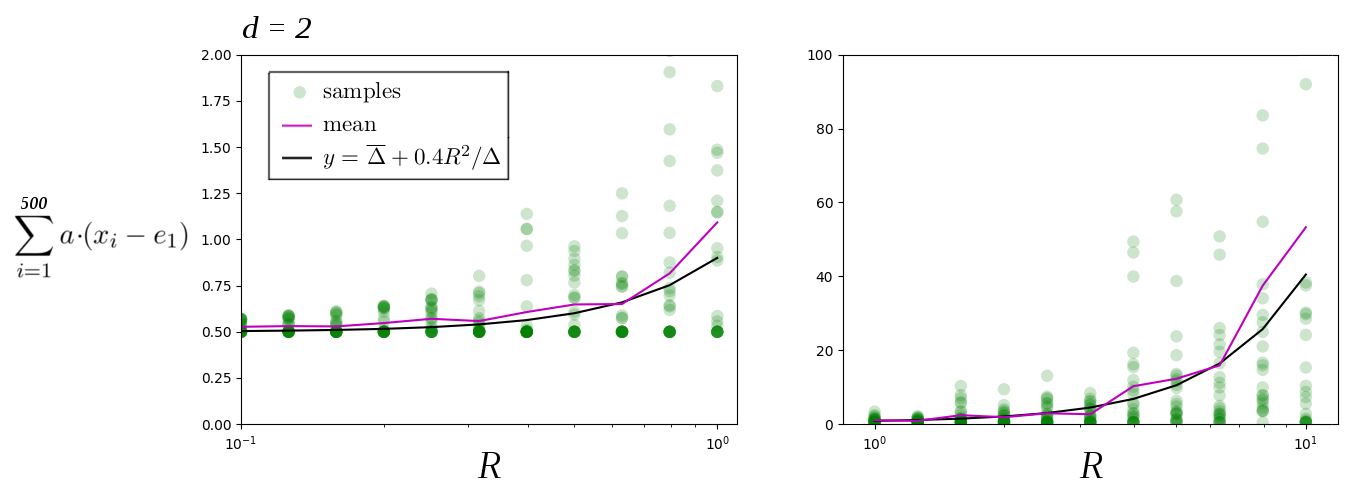}
  \end{center}
  \caption{Scatter plots of noise $R$ against pseudo-regret for $a = (0,1\ldots,1)$ and $d= 2$.  For each $R$-value we took 25 samples. Each sample ran for $500$ turns. The horizontal axes use a $\log$ scale. Some larger samples are excluded from the plot. }
  \end{figure}

To chose a good comparator consider the expression  $ 3/\eta^2 + 72R^2_2 e^{-1/2\eta^2 R^2}$ on the right-hand-side of Theorem 2. By setting $x=1/\eta^2$ and differentiating we find the minimiser $1/\eta^2 = 2R^2_2 \log 12$ gives minimum $6( 1 + \log 12)R^2_2 \simeq 21R^2$. On the other hand for $R \ge 1$ and the $\eta = 1$ used in the simulations we have $72R^2_2 e^{-1/2\eta^2 R^2} \ge (72/\sqrt e)R^2_2 \ge 43R^2$. These bounds seem too conservative as Figures 1 and 2 suggest  $\ol \DD + 0.4R^2/\DD$ for $\ol \DD= \frac{1}{d} \sum_{j=1}^d \DD_j $ is a more realistic bound.

 \begin{figure}[!h]
  \begin{center} 
    \includegraphics[width=1 \textwidth]{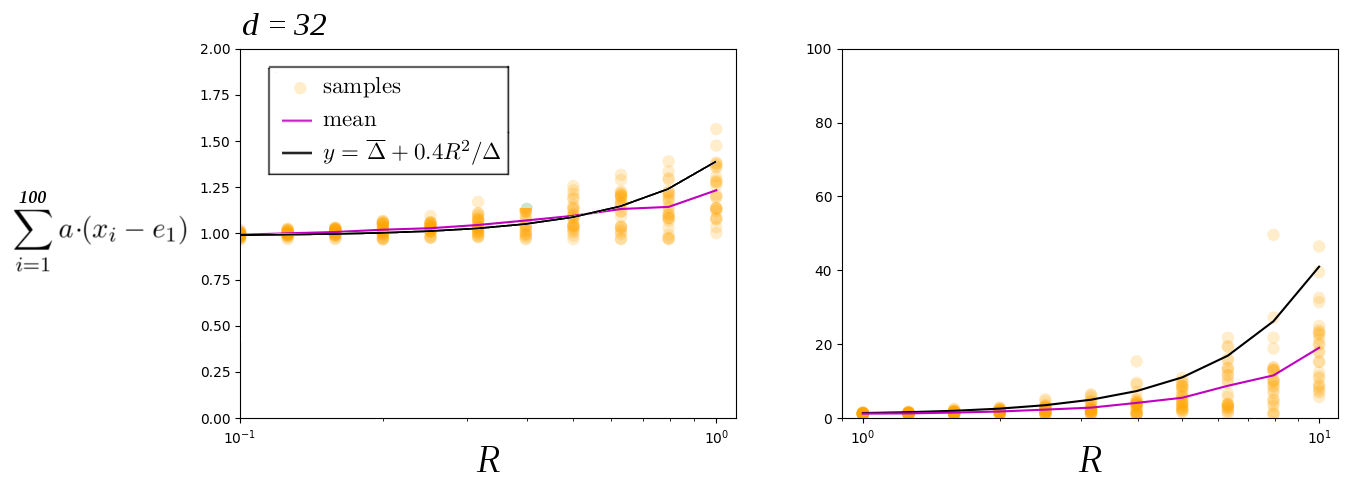}
  \end{center}
  \caption{Scatter plots of noise $R$ against pseudo-regret for $a = (0,1\ldots,1)$ and $d=32$.   For each $R$-value we took $25$ samples.  Each sample ran for $100$ turns. The horizontal axes use a $\log$ scale.}
  \end{figure}

Figures 3 and 4 also suggests higher dimensions {\it regularise} the data, lowering the   mean and significantly lowering the variance.  
Another observation is that $-$ even for large noise levels $-$ the behaviour seems to stabilise faster than the analysis suggests. Similar to  (\ref{stabilises}) we have for $n_0$  sufficiently high the bound:

\begin{align*}
 \Ex\left [	\sum_{i=n_0}^\infty a \cdot(x_i -e_1) \right] \le 2 \DD(2)   \sum_{i=n_0}^N e^{-\GG(2) n} + 2\sum_{k=3}^d \sum_{i=n_0}^\infty(\DD(k) -\DD(k-1))e^{-\GG(k) n}
\end{align*}  
for $\GG(k) = \frac{18\DD(k)^2}{R^2}$. 
In Figures 3 and 4  we have $\DD(k) = 1$ and
$R^2_2 =100$ and the second sum vanishes.  Replace the sum with an integral to see the right-hand-side is approximately $\frac{2} {\GG(2)}e^{-\GG(k) n_0} = 3600  e^{-n_0/1800} $.  This suggests we must wait until the order of turn $n_0 = 1800 \log 3600 \simeq 15000$ before the behaviour stabilises.  However Figures 3 and 4 suggest $N=500$ turns is enough for low dimensions and $N=100$ for higher dimensions.

 \begin{figure}[!h]
  \begin{center} 
    \includegraphics[width=1 \textwidth]{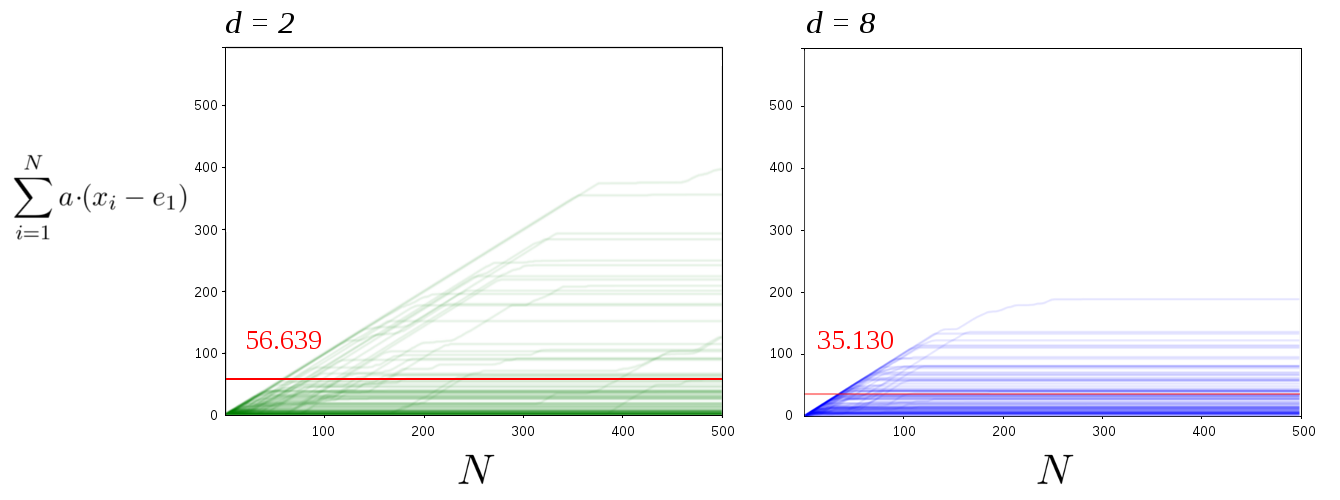}
  \end{center}
  \caption{Simultaneous line plots of 100 instances with $a = (0,1,\ldots,1)$ and $R=10$. Each instance ran for $500$ turns. The  red line is the average of  $\sum_{i=1}^{500} a \cdot(x_i -e_1)$ over the $100$ instances.}
  \end{figure}

 \begin{figure}[!h]
  \begin{center} 
    \includegraphics[width=1 \textwidth]{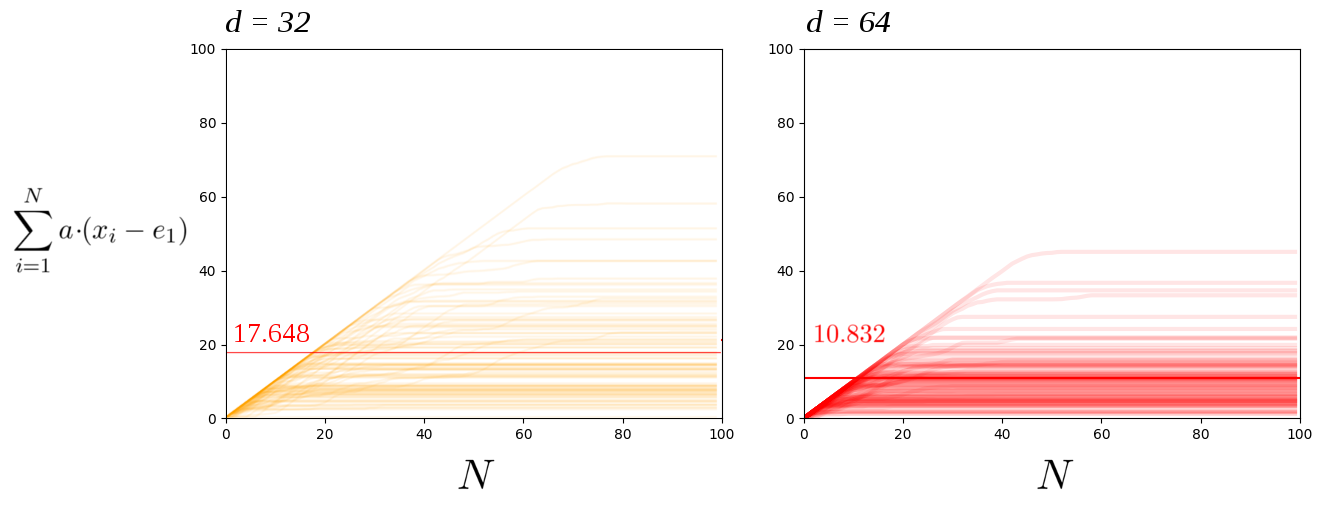}
  \end{center}
  \caption{Simultaneous line plots of 100 instances with $a = (0,1,\ldots,1)$ and $R=10$. Each instance ran for $100$ turns. The  red line is the average of  $\sum_{i=1}^{100} a \cdot(x_i -e_1)$ over the $100$ instances.}
  \end{figure}

The above simulations use $a=(0,1,\ldots, 1)$ because all other expectations we tried gave better performance. Two extreme cases are  $a=(0,1,2,\ldots, d-1)$ and $a = (0,\ldots, 0,1)$. The first gives moderately better performance in the long-run: The large cost on turn $1$ and differences between arms makes the pseudo-regret stabilise faster and gives  a steeper {\it shoulder} to the graph. The second gives significantly better performance.

 \begin{figure}[!h]
  \begin{center} 
    \includegraphics[width=1 \textwidth]{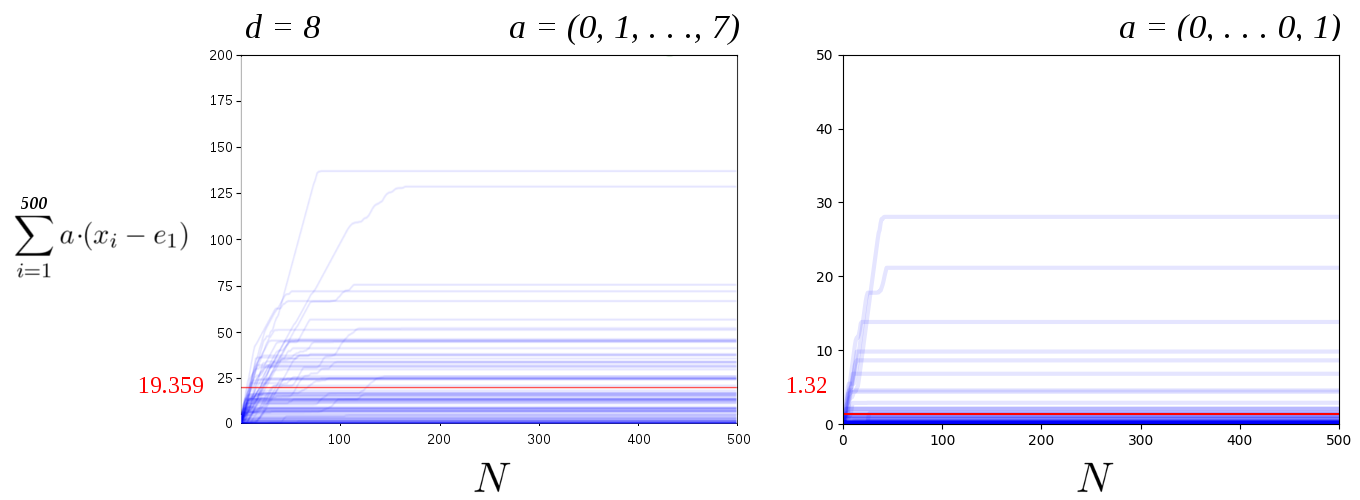}
  \end{center}
  \caption{Simultaneous line plots of 100 instances with $d=8$ and  $R=10$. Each instance ran for $500$ turns. The  red line is the average of  $\sum_{i=1}^{500} a \cdot(x_i -e_1)$ over the $100$ instances.}
  \end{figure}
 
  






\section*{Acknowledgements}
 This work was supported by Science Foundation Ireland grant 16/IA/4610.

\newpage





\section*{Appendix A: Regret in the General Setting}

\noindent Here we give the proof the subgradient algorithm with suitable parameter   has regret $O \big (L\sqrt N \big)$. The proof uses the   techniques from \cite{Purple1} modified slightly to not mention the time horizon. \\

\noindent{\bf Theorem \ref{worstcase}} {For cost vectors $b_1.b_2,\ldots, b_N$ with all $\|b_i\| \le L$ Algorithm 1 with parameter $\eta>0$ has regret satisfying
	\begin{align*}
	\sum_{i=1}^N b_i \cdot (x_{i} - x^*) \le   LD + \left (\frac{1}{2 \eta}\|\X\|^2 +  2\eta L^2_2 \right)  \sqrt {N} \end{align*}
	for $\|\X\| = \max\{\|x\|: x \in \X\}$ and $D$ the diameter of $\X$. In particular for $\X=\S$ and $\eta = 1/2L$ we have
	\begin{align*}
	\sum_{i=1}^N b_i \cdot (x_{i} - x^*) \le \sqrt 2 L_2 + 2 L_2  \sqrt N   .\end{align*}}

 \begin{proof}For $n >1$  define the functions $R_{n}(x) = \frac{\sqrt {n-1}}{2 \eta} \|x\|^2$. First we show each $x_{n}$ is the unique minimiser of $\sum_{i=1}^{n-1} b_i+ R_n(x)  $. Since rescaling by a positive constant does not change the minimisers the function has the same minimisers as 
 
 \begin{align}\|x\|^2 + \frac{2 \eta}{\sqrt {n-1}}  \sum_{i=1}^{n-1} b_i \cdot x \, = \,  \Big \|\,x + \frac{ \eta}{\sqrt {n-1}}  \sum_{i=1}^{n-1} b_i\,\Big \|^2 - \frac{ \eta^2}{n-1} \bigg (  \sum_{i=1}^{n-1}  b_i \bigg)^2\label{gather}
 \end{align}

 	Since the last term is constant the above has global minimum at $x = -\frac{ \eta}{\sqrt {n-1}}  \sum_{i=1}^{n-1}  b_i$. This is the point $y_{n}$ in the algorithm description. Lemma \ref{projection} says the minimum on $\X$ is the projection of the global minimum. Namely the point $x_{n}= P_\X(y_{n})$ as required. Now define the functions
 	$$Q_2(x)   = R_n(x) +  b_1 \cdot x +b_2 \cdot x \qquad Q_{n}(x) = R_n(x) - R_{n-1}(x) + b_n \cdot x \qquad \text{ for } n >2.$$  Clearly each $\sum_{i=2}^n Q_i  = \sum_{i=1}^n b_i  \cdot x + R_{n }(x)$. 
 	Lemma 3.1 of \cite{CBGames} says
 	$ \sum_{i=2}^N  Q_i  (z_i) \le \sum_{i=2}^N  Q_i  (x^*)$
 	where $z_n$ are any minimisers of $\sum_{i=2}^n   Q_i$ over  $\X$ and $x^* \in \X$ is arbitrary. Expanding both sides we get
 	\begin{align*}
 	 b_1  \cdot z_2 + \sum_{i=2}^{N} b_i \cdot z_i +  \frac{1}{2 \eta }\sum_{i=2}^{N} (\sqrt {n-1} - \sqrt {n-2})\|z_i\|^2 \le\sum_{i=1}^{N} b_i \cdot x^* +  \frac{\sqrt N}{2 \eta }\|x^*\|^2.
 	\end{align*}
 	Since the second sum is nonnegative we can neglect it. Bringing terms to the left and using $\|x^*\| \le \|\X\|$ we get
 	\begin{align*}
 	 b_1  \cdot (z_2-x^*) + \sum_{i=2}^{N} b_i \cdot (z_i-x^*)  \le    \frac{\sqrt N}{2 \eta }\|\X \|^2.
 	\end{align*} To get regret on the left-hand-side add $b_1 \cdot (x_1-z_2)+  \sum_{i=2}^N    b_i \cdot (x_i - z_i)  $ to both sides to get
 	\begin{align}
 	 \sum_{i=1}^N    b_i \cdot (x_i - x^*)   &\le  \frac{\sqrt {N}}{2 \eta}\|\X\|^2 + b_1 \cdot (x_1-z_2)+ \sum_{i=2}^N    b_i \cdot (x_i - z_i)\notag \\   &\le  \frac{\sqrt {N}}{2 \eta}\|\X\|^2 + LD + \sum_{i=2}^N    b_i \cdot (x_i - z_i)  \label{ApA}
 	\end{align}
 	
 	for $D$ the diameter of $\X$. 	To bound the sum on the right recall $z_n$ minimises $\sum_{i=2}^n Q_i(x) =  R_{n}(x)+ \Sum b_i \cdot x $. Similar to (\ref{gather}) we have 
 	$z_n = P_\X \left (- \frac{\eta}{\sqrt {n-1}}   \Sum b_i  \right)$. By definition  $x_n = P_\X \left (- \frac{\eta}{\sqrt {n-1}}   \sum_{i=1}^{n-1} b_i  \right)$
 	and so \begin{align*}\|x_n- z_n\|  &= \left \| P_\X \left (- \frac{\eta}{\sqrt {n-1}}   \Sum b_i  \right) - P_\X \left (- \frac{\eta}{\sqrt {n-1}}   \sum_{i=1}^{n-1} b_i  \right)\right \|\\  
 	& \le   \left \|   \frac{\eta}{\sqrt {n-1}}   \Sum b_i      - \frac{\eta}{\sqrt {n-1}}   \sum_{i=1}^{n-1} b_i \right \| = \frac{\eta}{\sqrt {n-1}} \|b_n\|\le \frac{\eta L}{\sqrt {n-1}}  
 	\end{align*} 
 	where the inequality uses Theorem 23 of \cite{NonExpansive}. By Cauchy-Schwarz  the sum in  (\ref{ApA}) is at most
 	\begin{align*}   \sum_{i=2}^N    \|b_i\|\| x_i - z_i\|  \le \sum_{i=2}^N    \frac{\eta L^2_2}{\sqrt {i-1}} =\sum_{i=1}^{N-1}    \frac{\eta L^2_2}{\sqrt {i }}     \le   \eta L^2_2 \int_{0}^{N} \frac{dx}{\sqrt x} = 2\eta L^2_2 \sqrt N \end{align*}
 	and (\ref{ApA}) simplifies to
 	\begin{align}\label{ApA1}
 	 \sum_{i=1}^N    b_i \cdot (x_i - x^*)   &\le    LD + \frac{\sqrt {N}}{2 \eta}\|\X\|^2 +   2 \eta L^2_2  \sqrt N   .
 	\end{align}
 	For parameter   $\eta = \frac{\|\X\|}{2L}$ the above is $LD + 2 \|\X\|L_2 \sqrt N $. For $\X$ the simplex $\|\X\| =1$ and $D= \sqrt 2$ and we get $\sqrt 2 L_2 + 2 L_2  \sqrt N $.
\end{proof}
\vspace{-5mm}
 
\section*{Appendix B: Convex Geometry}  Here we prove the convex geometry lemmas needed for the main analysis. The first is well known. It says the contrained minimum of a quadratic function is the projection of the global minimum.

\begin{lemma} \normalfont\label{projection}
	Suppose $\AA \ge 0$ and $F(x) = \AA \|x-v\|^2 + w$ is a quadratic function on $\RR^d$ and $\X \subset \RR^d$ convex. Then $\am\{ F(x): x \in \X\} = P_\X(v)$.
\end{lemma}

\begin{proof}By definition $P_\X(v) = \am\{\|x-v\|^2: x \in \X\}$. Since positive rescaling   and adding a constant does not change the minimisers we have $P_\X(v) = \am\{\AA\|x-v\|^2 + w: x \in \X\} = $ $\am\{F(x): x \in \X\}$.
\end{proof}  
Lemma \ref{goal}  is used to show a point projects onto the optimal vertex of the simplex. 
 
\begin{lemma} \normalfont  Suppose $w \in \RR^d$ has   $w_k > w_\ell$. Then for $u = P_\S(w)$  we have $u_k \ge u_\ell$.
\end{lemma}
\begin{proof}
By definition $\ds \min_{x \in \S} \Sumd (w_j-x_j)^2 =   \sum_{j \ne k,\ell}^d (w_j-u_j)^2 + (w_k-u_k)^2+(w_\ell-u_\ell)^2 .$ 
For a contradiction suppose $u_\ell > u_k$. 
We claim the above gets strictly smaller if we swap components $u_k$ and $u_\ell$. Since this swap gives a new point on the simplex it contradicts the definition of $u$ as a minimiser. 
To complete the proof write. 
$$   (w_k-u_k)^2+(w_\ell-u_\ell)^2 = (w_k^2 + w_\ell^2 + u_k^2 + u_\ell^2) -2(w_\ell u_\ell +w_k u_k ).$$ The first term is invariant under exchanging $u_\ell$ and $u_k$. 
For the second term we must show $w_\ell u_k + w_k u_\ell \ge w_\ell u_\ell +w_k u_k$. This is equivalent to  $w_k(u_\ell-u_k) > w_l(u_\ell-u_k)$ which holds since $u_\ell-u_k >0$ and $w_k > w_\ell$.
\end{proof}

\noindent{\bf Lemma \ref{goal}}  Suppose $w \in \RR^d$ has two coordinates $k,\ell$ with $w_k - w_\ell \ge 1$. Then $P_\S(w)$ has $\ell$-coordinate zero.\vspace{5mm}

\begin{proof} Like before write $u = \P_\S(w)$ and recall  $$\ds \min_{x \in \S} \Sumd (w_j-x_j)^2 =   \sum_{j \ne k,\ell}^d (w_j-u_j)^2 + (w_k-u_k)^2+(w_\ell-u_\ell)^2 $$ Write $U = u_k+u_\ell$. Clearly $u$ minimises $(w_k-u_k)^2 + (w_\ell-u_\ell)^2$ over $u_k+u_\ell=U$. In other words $u$ minimises $(w_k-U +u_\ell)^2 + (w_\ell-u_\ell)^2$ over $u_\ell \in [0,U]$. By differentiating we see the minimum over $u_\ell \in \RR$ is $u_\ell = \frac{U+(u_\ell-u_k)}{2} \le \frac{U-1}{2} \le 0$. Since the function is a quadratic it is increasing on $ [0,U]$ and the minimum is $u_\ell=0$ as required.
\end{proof}  
\section*{Appendix C: Probability}

Our main concentration result is due to \cite{GoodAH}.
\begin{theorem} \normalfont\label{Pinelis} (Pinelis Theorem 3.5) Suppose the martingale $f_1,\ldots, f_n$ takes values in the $(2,D)$-smooth Banach space $(E,\|\cdot\|)$. Suppose we have $ \|f_1\|_\infty^2 + \sum_{i=2}^n \|f_i - f_{i-1}\|_\infty^2 \le b^2$ for some constant $b$. Then for all $t \ge 0$ we have
	$$P \left( \max\{\|f_1\|, \ldots, \|f_n\|\} \ge t \right) \le 2\exp \left( -\frac{t^2}{2D^2 b^2}\right).$$
\end{theorem}

Here $\|f\|_\infty = \max\{\|f(x)\|: x \in \Omega\}$ is the $\sup$ norm taken over the probability space. The Banach space $(E, \|\cdot\|)$ is called $(2,D)$-smoooth to mean $\|x+y\|^2 + \|x-y\|^2 \le 2\|x\|^2 + 2D^2\|x\|^2$ for all $x,y \in E$. The fact that $\RR^d$ is $(2,D)$-smooth is sometimes called the { \it parallelogram law}. 

See for example \cite{Bill} Section 35 for the definition of a martingale and martingale difference sequence. 
It is well known that if $a_1,a_2,\ldots$ are i.i.d with $\Ex[a_i]=a$ then $f_n = \sum_{i=1}^n(a_i-a)$ defines a martingale. If $\|a_i - a\| \le R$ then taking $b^2 = nR^2$ and $t = tn $ in the Pinelis theorem we have the following.

\begin{theorem} \normalfont\label{Pinelis2}Suppose the i.i.d sequence $a_1,a_2,\ldots$ takes values in $\RR^d$. Suppose for $\Ex[a_i] = a$ we have $\|a_i -a\| \le R$. Then for each $t \ge 0$ we have
	$$P \left( \frac{1}{n}\Big \|\sum_{i=1}^n(a_i - a) \Big\| \ge t \right) \le 2\exp \left( -\frac{t^2}{2 R^2}n\right).$$
\end{theorem}

The following fact about computing the expectation in terms of the CDF is well-known. But we were unable to find a suitably general proof in the literature.

\begin{lemma} \normalfont\label{CDF}
 Suppose $X$ is a real-valued random variable. Then \[\Ex[X] = \int_0^\infty P(X>x)dx - \int_{-\infty}^0 P(X\le x)dx.\] In particular \[\Ex[X] \le \int_0^\infty P(X>x)dx .\]

 \end{lemma}

\begin{proof} First assume $X$ takes only positive values. The second integral vanishes and we can write the first as 
\begin{align*}
\int_0^\infty P(X > x)dx = \int_0^\infty  \Ex_y  \mkern-5mu \left [ \1_{X(y) > x} (y) \right ] dx = \Ex_y  \mkern-5mu \left [ \int_0^\infty \1 _{X(y) > x} (y)dx \right].
\end{align*}
For fixed $y$ define the function $g(x) = \1 _{X(y) > x} (y)$. We have $g(x)= 1$ for all $x > X(y)$ and $g(x)= 0$ elsewhere. Since $X(y)$ is nonnegative that means $g(x)$ is the indicator function of $[0,X(y))$. It follows the inner integral equals $X(y)$ and the above becomes $\Ex_y [X(y)]=\Ex[X]$. Observe the above also holds if we assume $X$ takes only nonnegative values and replace $P(X > x)$ with $P(X \ge x)$.

For a general random variable we can write $X=X^+ +X^-$ where $X^+$ takes only nonnegative values and $X^-$ only nonpositive values, and at each point one of $X^+$ or $X^-$ is zero. Since $-X^-$ is nonnegative we have already shown 
\begin{align*}
\Ex[-X^-] = \int_0^\infty P(-X^- \ge x)dx = \int_0^\infty P(X^- \le -x)dx = \int_{- \infty}^0 P(X^- \le x)dx
\end{align*}
The left-hand-side is $-\Ex[X^-]$. By construction $P(X^- \le x) = P(X \le x)$ for each $x \le 0$. Hence the right-hand-side is $\int_{- \infty}^0 P(X \le x)dx$. Finally write
\begin{align*}
\Ex[X] = \Ex[X^+] + \Ex[X^-] =\Ex[X^+] - \Ex[-X^-] = \int_0^\infty P(X>x)dx - \int_{-\infty}^0 P(X\le x)dx.
\end{align*}
\end{proof} 
At one stage we use the scalar Azuma-Hoeffding inequality  to get one-sided bounds and avoid the leading factor of $2$ in the Pinelis Theorem. See \cite{MIT2} Lecture 12  for proof.

\begin{theorem} \normalfont(Scalar Azuma-Hoeffding) \label{Hoeffding}Suppose $X_1,X_2\ldots,  $ is a real-valued Martingale difference sequence with each $ |X_i|  \le R$. For all $n \in \NN$ and $t \in \RR$ we have $$P \left(   \sum_{i=1}^n X_i   \ge  t \right) \le  \exp \left( -\frac{t^2}{2 R^2} n\right).$$
 
\end{theorem}

The scalar Azuma-Hoeffding Inequality is used in Section \ref{tail}. To that end we need the following lemma showns a certain sequence of random variables that appears in that section is indeed a martingale. 
\begin{lemma} \normalfont\label{martingale} Let $a_1,a_2,\ldots$ be an i.i.d sequence of cost vectors and $x_1,x_2,\ldots$ the actions of Algorithm 1. The random variables $X_i = (a-a_i) \cdot( x_i-x^*)$ define  a martingale difference sequence  with respect to the filtration generated by $a_1, a_2,\ldots $. 
\end{lemma}

\begin{proof}
 We must show each $\Ex[X_n | a_1, \ldots, a_{n-1}] =0$. That means for each set $U = a_1^{-1}(U_1) \cap \ldots \cap a_{n-1}^{-1}(U_{n-1}) $ in the algebra generated by $a_1,a_2,\ldots a_{n-1}$ we have  $\int_U X_n dP =0$.
 To that end write each $B(i) = a_i^{-1}(U_i)$ and observe the indicator $\1_{B(i)}$ is a measurable function of $a_1,\ldots, a_{n-1}$. Now write
	
	\begin{align*} \int_U X_n dP = \int_U  (a-a_{n }) \cdot( x_n -x^*) dP = \int   (a-a_{n }) \cdot( x_n-x^*)  \1_{B(1)} \cdot \ldots \cdot \1_{B(n-1)} dP.
	\end{align*}
	
	Recall $x_n$ is a function of $a_1,\ldots, a_{n-1}$. Since all $a_i$ are independent we can distribute to get 
	\begin{align*} \int_U  (a-a_{n }) \cdot( x_n -x^*) dP = \int   (a-a_{n}) dP \cdot \int ( x_n -x^*)  \1_{B(1)} \cdot \ldots \cdot \1_{B(n-1)} dP.
	\end{align*}
	
	Since $\Ex[a_n]=a$ the above is zero as required. 
\end{proof} 


	
	
	

\end{document}